\documentclass[11pt]{amsart}
\usepackage[left=30truemm,right=30truemm,top=30truemm,bottom=30truemm]{geometry}
\usepackage{amsmath, amssymb}
\usepackage{mathrsfs} 
\usepackage{mathtools} 
\usepackage{bm}
\usepackage{xcolor}
\usepackage[all]{xy} 
\xyoption{poly}

\usepackage{hyperref}
\definecolor{yelloworange}{RGB}{255,148,0}
\hypersetup{
	colorlinks=true, 
	citecolor=blue, 
	linkcolor=brown, 
	urlcolor=yelloworange, 
	breaklinks=true, 
}

\makeatletter
\@addtoreset{figure}{section}
\@addtoreset{table}{section}
\makeatother

\usepackage{amsthm}
\usepackage[capitalize, nameinlink]{cleveref}
\theoremstyle{plain}
\newtheorem{theorem}{Theorem}[section]
 
\newtheorem{lemma}[theorem]{Lemma} 
\newtheorem{proposition}[theorem]{Proposition} 
\newtheorem{thmintro}{Theorem}
\crefname{thmintro}{Theorem}{Theorems}

\theoremstyle{definition}

\theoremstyle{remark}
\newtheorem{remark}[theorem]{Remark} 
\newtheorem{example}[theorem]{Example}

\usepackage{etoolbox} 
\newcommand{\bAlphabet}{
  \def\do##1{\expandafter\newcommand\csname b##1\endcsname{\mathbb{##1}}}
  \docsvlist{C,F,P,Q,R,Z}
}
\bAlphabet
\newcommand{\cAlphabet}{
  \def\do##1{\expandafter\newcommand\csname c##1\endcsname{\mathcal{##1}}}
  \docsvlist{A,D,O,T}
}
\cAlphabet
\newcommand{\frakAlphabet}{
  \def\do##1{\expandafter\newcommand\csname frak##1\endcsname{\mathfrak{##1}}}
  \docsvlist{m}
}
\frakAlphabet

\newcommand{\id}{\mathrm{id}}
\newcommand{\Aut}{\operatorname{Aut}}

\newcommand{\GL}{\operatorname{GL}}
\newcommand{\NS}{\operatorname{NS}}
\newcommand{\M}{\operatorname{M}}

\newcommand{\PGL}{\operatorname{PGL}}

\newcommand{\SL}{\operatorname{SL}}
\newcommand{\SO}{\operatorname{SO}}
\newcommand{\Tr}{\operatorname{Tr}}

\title{Unboundedness of fixed point multiplicities on a K3 surface}
\date{November 22, 2025}

\author{Kenji Hashimoto}
\address{Kenji Hashimoto}
\email{kenji.hashimoto.math@gmail.com}

\author{Yuta Takada}
\address{Yuta Takada \newline 
\hspace*{2.7mm}
Mathematical Sciences, The University of Tokyo, 
Tokyo 153-8914, Japan; JSPS Research Fellow. \newline
\hspace*{3.5mm}
Department of Mathematics, 
Saarland University, Saarbr\"ucken 66123, Germany.}
\email{ytakada@ms.u-tokyo.ac.jp}

\begin{document}
\begin{abstract}
We exhibit automorphisms of a certain K3 surface in $\bP^1\times \bP^1 \times \bP^1$ 
with an isolated fixed point at which the induced action on the stalk of the structure sheaf 
is arbitrarily close to the identity. This implies that the multiplicities of these automorphisms 
at the fixed point can be arbitrarily large. As another application, we show that 
the intersection multiplicity of two isomorphic curves at a point can be arbitrarily large
on this K3 surface.
\end{abstract}
\maketitle

\section{Introduction}
Let $X$ be a complex manifold and $p\in X$ a point. 
Let $\gamma:X\to X$ be an automorphism that fixes the point $p$. 
The \textit{multiplicity} $\nu_p(\gamma)$ of $\gamma$ at $p$ is defined as  
\[ \nu_p(\gamma) \coloneqq \dim \left(\widehat\cO_p/ (\gamma^*_pf - f\mid f\in  \widehat\cO_p)\right) \]
where $\widehat\cO_p$ is the completion of the stalk of the structure sheaf at $p$, and 
$\gamma^*_p:\widehat\cO_p\to \widehat\cO_p$ is the induced homomorphism. 
If $(x_1, \ldots, x_d)$ is a local coordinate system centered at $p$, and $\gamma$ is expressed as 
$\gamma(x_1, \ldots, x_d) = (\gamma_1(x_1, \ldots, x_d), \ldots, \gamma_d(x_1, \ldots, x_d))$, then 
\[ \nu_p(\gamma) = \dim \left(\bC[[ x_1, \ldots, x_d]]
/ (\gamma_1(x_1, \ldots, x_d) - x_1, \ldots, \gamma_d(x_1, \ldots, x_d) - x_d) \right). \]
The multiplicity $\nu_p(\gamma)$ is finite if and only if $p$ is an isolated fixed point of $\gamma$. 
It is known that $\nu_p(\gamma^n)$ is bounded as a function of $n$ 
when $p$ is an isolated fixed point of $\gamma^n$ for every $n\geq 1$; 
see \cite{Shub-Sullivan1974}. 
This raises the question of whether $\nu_p(\gamma)$ is bounded when 
$\gamma$ ranges over automorphisms for which $p$ is an isolated fixed point. 
Let us consider the set  
\[ M(X,p)\coloneqq \{ \nu_p(\gamma) \mid \text{ $\gamma$ is an automorphism of $X$ 
for which $p$ is an isolated fixed point} \} \]
under the global assumption that $X$ is compact. 

\begin{example}\label{ex:ofgeneraltype}
Suppose that the automorphism group $\Aut(X)$ is a finite group (as in the case, for example, 
when $X$ is of general type; see \cite{Matsumura1963}). 
In this case, we have $M(X,p)=\{ 1 \}$ for any $p$ that is an isolated fixed point of some 
automorphism of $X$, since every automorphism of finite order is locally linearizable in a 
neighborhood of each of its fixed points (see e.g. \cite[Chap.\ I, \S 6]{Bochner-Martin1948}).
\end{example}

Both projective spaces $\bP^d$ and complex tori provide examples of complex manifolds $X$ 
such that $|\Aut(X)|$ is infinite while $M(X,p)$ is bounded for any $p\in X$, 
as explained below. 

\begin{example}\label{ex:bP^d}
Let $X = \bP^d$. The automorphism group $\Aut(\bP^d)$ is naturally identified 
with the projective general linear group $\PGL_{d+1}(\bC)$. Since the action of
$\PGL_{d+1}(\bC)$ is transitive, the set $M(\bP^d, p)$ does not depend on the 
choice of the point $p$. 
Let $p$ be the point given by $[x_0:x_1:\cdots:x_d]=[1:0:\cdots:0]$, and let 
$[A]\in \PGL_{d+1}(\bC) = \Aut(\bP^d)$ be an automorphism fixing $p$, 
where $[A]$ denotes the class of a matrix $A\in \GL_{d+1}(\bC)$. 
We may assume without loss of generality that $A$ is in Jordan normal form.  
Furthermore, we may assume that the eigenvalue of the first Jordan block 
is $1$ after multiplying $A$ by a scalar. 
Suppose that $p$ is an isolated fixed point of $[A]$. 
Then, there is no Jordan block with eigenvalue $1$ except for the 
first one, and one can see that the multiplicity $\nu_p([A])$ is equal to 
the size of the first Jordan block. 
Hence, we obtain $M(\bP^d, p) = \{ 1, 2, \ldots, d+1 \}$. 
\end{example}

\begin{example}\label{ex:torus}
Let $X = \bC^d/\Lambda$ be a torus, where $\Lambda$ is the $\bZ$-span of an $\bR$-basis 
of $\bC^d$. Then, $X$ acts on itself as translations, and any 
$A\in G \coloneqq\{ A\in \GL_d(\bC)\mid A\Lambda = \Lambda \}$ 
induces an automorphism of $X$. It is known that $\Aut(X) = G\ltimes X$. 
Since $X$ acts on itself transitively, the set $M(X,p)$ does not depend on the choice of 
the point $p$. We have $M(X, 0) = \{1\}$ because if $\gamma \in \Aut(X)$ fixes the 
point $0$ then it is induced by a linear transformation of $\bC^d$. 
\end{example}

Since any compact curve i.e., one-dimensional complex manifold $X$ is either $\bP^1$, 
a torus, or a curve of general type, depending on its genus, 
\cref{ex:ofgeneraltype,ex:bP^d,ex:torus} show that the set $M(X,p)$ is bounded 
for any point $p\in X$. This can also be deduced from the Lefschetz fixed point formula. 

Does there exist an example in dimension two for which $M(X,p)$ is unbounded? 
This article gives an affirmative answer in \cref{th:mainA} by constructing a K3 surface $X$ 
with a point $p$ such that the set  $M(X,p)$ is unbounded. 
In fact, we prove a stronger statement as follows. 
Let $\frakm_p$ denote the maximal ideal of $\widehat\cO_p$. 
For an integer $n\geq 1$ and an automorphism $\gamma$ of $X$, 
we say that $\gamma^*_p:\widehat\cO_{p}\to \widehat\cO_{p}$ 
is congruent to $\id$ modulo $\mathfrak{m}_p^n$, 
written $\gamma^*_p \equiv \id \bmod \frakm_p^n$, 
if  $\gamma^*_pf \equiv f \bmod \frakm_p^n$ for any $f\in \widehat\cO_p$. 
Then, for any $n\geq 1$, there exists an automorphism $\gamma$ of $X$ such that
$p$ is an isolated fixed point of $\gamma$ and $\gamma^* \equiv  \id \bmod \frakm_p^n$. 

Our K3 surface is constructed as a smooth surface in $\bP^1\times \bP^1 \times \bP^1$. 
Let $\bP^1_x$, $\bP^1_y$, and $\bP^1_z$ denote the projective lines with homogeneous coordinates 
$(x_0: x_1)$, $(y_0: y_1)$, and $(z_0: z_1)$, respectively, and put 
$Z = \bP^1_x\times \bP^1_y \times \bP^1_z$. 
Let $X\subset Z$ be the surface defined by the tri-homogeneous polynomial 
\[ 
(x_0^2y_1^2 + x_1^2y_0^2) z_0^2
+ (x_0^2 y_0^2 + x_1^2 y_1^2)z_0 z_1 
+ (x_0^2y_1^2 + x_0x_1y_0y_1 + x_1^2y_0^2) z_1^2  
\]
of tri-degree $(2,2,2)$. One can check that $X$ is smooth by the Jacobian criterion; 
see also \cref{prop:Xissmooth}. 
Let $p\in X$ be the point given by $(x_0:x_1, y_0:y_1, z_0:z_1) = (1:0, 1:0, 1:0)$. 

\begin{thmintro}\label{th:mainA}
Let $X$ and $p$ be as above. 
For any $n\geq 1$, there exists an automorphism $\gamma$ of $X$ that fixes $p$ such that
\begin{enumerate}
\item $\gamma$ preserves no irreducible curve, which implies that
$p$ is an isolated fixed point of $\gamma$; and 
\item the induced ring homomorphism $\gamma^*_p:\widehat\cO_{p}\to \widehat\cO_{p}$ 
is congruent to $\id$ modulo $\mathfrak{m}_{p}^n$. 
\end{enumerate}
In particular, the set 
\[ M(X,p) = \{\nu_p(\gamma) \mid \text{$\gamma$ is an automorphism of $X$ for which 
$p$ is an isolated fixed point} \} \]
has no upper bound. 
\end{thmintro}

In fact, we can obtain a one-parameter family of K3 surfaces whose 
generic member satisfies \cref{th:mainA}; see \cref{rmk:parameter}. 
The construction of $\gamma$ in \cref{th:mainA} is as follows. 
The restriction of the projection $Z \to \bP_y\times \bP_z$ to $X$ is a double cover, 
which defines an involution $\iota_x:X\to X$. 
Similarly, other two involutions $\iota_y$ and $\iota_z$ are defined. 
Let $\phi$ be a suitable power of $\iota_x\circ\iota_y$, and let 
$\psi = \iota_z\circ\iota_x\circ\iota_y\circ \iota_z$.  
These two automorphisms fix the point $p$. 
We recursively define automorphisms $\gamma_i$ of $X$ by 
$\gamma_0 \coloneqq \psi$ and $\gamma_{i} \coloneqq [\phi, \gamma_{i-1}]$ ($i\geq 1$), where 
$[\phi, \gamma_{i-1}] = \phi^{-1}\circ \gamma_{i-1}^{-1}\circ \phi\circ \gamma_{i-1}$ is 
the commutator. Then, $\gamma_n$ will be the desired automorphism.  
The defining polynomial of $X$ was chosen so that this construction works well. 

\begin{remark}
In his paper \cite{McMullen2002}, McMullen considered automorphisms $f$ of K3 surfaces 
that have an isolated fixed point $p$ such that $f$ acts on a neighborhood of $p$ as 
an irrational rotation. In this case, the multiplicity $\nu_p(f)$ must be $1$. 
By contrast, the automorphism $\gamma$ 
in \cref{th:mainA} is tangent to the identity and cannot be locally linearized. 
\end{remark}

\cref{th:mainA} yields another consequence. For two curves $C$ and $D$ on a surface $X$ 
through a point $p$, let $\mu_p(C\cdot D)$ denote the \textit{intersection multiplicity} 
of $C$ and $D$ at $p$: 
\[ \mu_p(C\cdot D) \coloneqq \dim \left(\widehat\cO_p/(I + J) \right) \]
where $I$ and $J$ are the ideals defining the germs of $C$ and $D$ at $p$, respectively. 

\begin{thmintro}\label{th:mainB}
Let $X$ and $p$ be as in \cref{th:mainA}. Let $C\subset X$ be a nonsingular irreducible 
curve through $p$. For any $n\geq 1$, there exists a curve $C'$ that is isomorphic to but 
distinct from $C$ such that $\mu_p(C\cdot C') \geq n$. 
\end{thmintro}

\cref{th:mainB} follows from \cref{th:mainA} by setting $C' = \gamma(C)$ or $\gamma^{-1}(C)$. 
However, we remark that the statement of \cref{th:mainB} is independent of automorphisms. 
An example of nonsingular irreducible curves through $p$ is the ramification curve of 
the double cover $X\to \bP^1_y\times \bP^1_z$, that is, the fixed point set of $\iota_x$. 
Another example is given as follows. 
The restriction of the projection $Z\to \bP^1_z$ to $X$ defines an elliptic fibration. 
The fiber of this fibration over $(z_0: z_1) = (1:0)$, 
\[ \{ (x_0:x_1, y_0:y_1, 1:0)\in \bP^1_x\times \bP^1_y\times \bP^1_z  
\mid x_0^2y_1^2 + x_1^2y_0^2 = 0 \}\subset X, \]
consists of two $(-2)$-curves. Each of them is a nonsingular irreducible curve through $p$. 

The organization of this article is as follows. In \S\ref{ss:commandmulti}, 
we provide some local arguments explaining why taking commutators yields
the desired automorphism in \cref{th:mainA}. 
In the rest of \cref{sec:pleliminaries}, we collect several lemmata and 
propositions that we will need later. 
In \cref{sec:K3}, after briefly recalling definitions and properties 
of K3 surfaces, we study K3 surfaces in 
$\bP^1\times \bP^1\times \bP^1$. 
\cref{sec:mainresults} is devoted to the proofs of 
\cref{th:mainA,th:mainB}.

\subsection*{Acknowledgments.}
The second author is supported by JSPS KAKENHI Grant Number JP24KJ0044. 

\section{Preliminaries}\label{sec:pleliminaries}
\subsection{Commutator and multiplicity}\label{ss:commandmulti}
The fixed point multiplicity is a local quantity, and in this subsection we focus on 
local arguments. Although similar discussions can be carried out in higher dimensions, 
we restrict ourselves to the two-dimensional case. 

Let $\cO_0$ and $\widehat\cO_0$ denote the local ring of holomorphic functions 
defined near $0\in \bC^2$ and its completion, respectively. 
The symbol $\frakm$ denotes the maximal ideal of $\widehat\cO_0$. 
Note that $\widehat\cO_0$ is isomorphic to the ring $\bC[[x,y]]$ of formal power series, and 
under this isomorphism, we have 
$\frakm^k = (x^iy^j\mid i, j \in \bZ_{\geq0}, i + j = k)$ for any $k\in \bZ_{\geq0}$. 

Let $G$ be the set of germs of biholomorphic mappings from neighborhoods of $0$ 
in $\bC^2$ to $\bC^2$ that send $0$ to $0$. The composition of mappings makes $G$ a group. 
Every $\gamma\in G$ induces a ring homomorphism $\gamma^*:\widehat\cO_0\to \widehat\cO_0$ 
naturally. As mentioned in Introduction, the multiplicity $\nu_0(\gamma)$ of $\gamma$ 
at $0$ is defined as 
\[ \nu_0(\gamma) \coloneqq \dim \left(\widehat\cO_0/ (\gamma^*f - f\mid f\in  \widehat\cO_0)\right).\]
The multiplicity $\nu_0(\gamma)$ is finite if and only if $0$ is an isolated fixed point of $\gamma$. 
For $\gamma\in G$ and an ideal $I \subset \widehat\cO_0$, we write $\gamma^* \equiv \id \bmod I$ if 
$\gamma^*f \equiv f \bmod I$ for any $f\in \widehat\cO_0$.  
For $\phi, \psi \in G$, the commutator $\phi^{-1}\circ\psi^{-1}\circ\phi\circ\psi$ 
is denoted by $[\phi, \psi]$.

\begin{lemma}\label{lem:commutatormodI}
Let $\phi, \psi \in G$ and $k\in \bZ_{\geq1}$. 
Assume that $\phi^* \equiv \id \bmod \frakm^2$ and $\psi^* \equiv \id \bmod \frakm^k$. 
Then, we have $[\phi, \psi]^* \equiv \id \bmod \frakm^{k+1}$.
\end{lemma}
\begin{proof}
If $k = 1$, 
since $\phi^*\equiv \id \bmod \frakm^2$, we have 
\[ 
[\phi, \psi]^* = (\phi^{-1}\circ\psi^{-1}\circ\phi\circ\psi)^* 
\equiv (\psi^{-1}\circ\psi)^* \equiv \id \mod \frakm^2. 
\]
Suppose that $k \geq 2$. Then $\phi$ and $\psi$ can be written as 
\[\begin{split}
\phi(x, y) 
&= \left(x + \phi'_1(x, y), y + \phi'_2(x, y)\right),\\ 
\psi(x, y) 
&= \left(x + \psi'_1(x, y), y + \psi'_2(x, y)\right), 
\end{split} 
\]
where each $\phi_i'(x,y)$ belongs to $\frakm^{2}$, while each $\psi_i'(x, y)$
belongs to $\frakm^k$. We have 
\[\begin{split}
(\phi\circ\psi)^*x
&= x + \psi_1'(x,y) + \phi_1'(x + \psi'_1(x, y), y + \psi'_2(x, y))\\
&\equiv x + \psi_1'(x,y) + \phi_1'(x, y) \mod \frakm^{k+1}
\end{split}
\]
and similarly, 
\[
(\phi\circ\psi)^*y
\equiv y + \psi_2'(x,y) + \phi_2'(x, y) \mod \frakm^{k+1}.  
\]
A similar calculation yields 
\[\begin{split}
(\psi\circ\phi)^*x
&\equiv x + \phi_1'(x,y) + \psi_1'(x, y)
\equiv (\phi\circ\psi)^*x \mod \frakm^{k+1}, \\
(\psi\circ\phi)^*y
&\equiv y + \phi_2'(x,y) + \psi_2'(x, y)
\equiv (\phi\circ\psi)^*y \mod \frakm^{k+1}. \\
\end{split}\]
These congruences show that 
\[ (\phi^{-1}\circ\psi^{-1}\circ\phi\circ\psi)^* 
= (\phi\circ\psi)^* \circ ((\psi\circ\phi)^*)^{-1}
\equiv \id \mod \frakm^{k+1},  
\]
and the proof is complete. 
\end{proof}

\begin{remark}
Let $G_k := \{ \gamma \in G \bigm| \gamma^* \equiv \id \bmod \frakm^{k+1} \}$.
\Cref{lem:commutatormodI} means that $[G_1, G_k] \leq G_{k+1}$ for any $k\geq 1$. 
In other words, the sequence 
\[    G_1 \rhd G_2 \rhd G_3 \rhd \cdots \]
is a central series of $G_1$ of infinite length. 
\end{remark}

As in \cref{lem:commutatormodI}, let $\phi$ and $\psi \in G$ be germs such that 
$\phi^* \equiv \id \mod \frakm^2$ and $\psi^* \equiv \id \mod \frakm^k$ for some $k\in \bZ_{\geq1}$. 
Let $\Gamma$ be the subgroup of $G$ generated by $\phi$ and $\psi$. 
We recursively define biholomorphic mapping germs $\gamma_i\in \Gamma$ ($i\geq 0$) by 
\[
\gamma_0 \coloneqq \psi, \quad 
\gamma_{i} \coloneqq [\phi, \gamma_{i-1}] \quad(i\geq 1). 
\] 

\begin{proposition}\label{prop:idmodm_n}
Suppose that $\phi$ and $\psi$ satisfy no nontrivial group-theoretic relation,
i.e., $\Gamma$ is a free group of rank $2$. 
Then, for any integer $n\geq k$, we have $\gamma_{n-k} \neq \id$ and 
$\gamma_{n-k}^* \equiv \id \bmod \frakm^{n}$. 
In this case, we have $\nu_0(\gamma_{n-k})\geq n(n+1)/2$ (possibly infinite). 
\end{proposition}
\begin{proof}
Note that $\gamma_i \neq \id$ for any $i\geq 0$ by the assumption. 
Let $n\geq k$. We have $\gamma_0^* = \psi^* \equiv \id \bmod \frakm^k$ and get 
\[\begin{split}
\gamma_1^* &\equiv \id \mod \frakm^{k+1}, \\
\gamma_2^* &\equiv \id \mod \frakm^{k+2}, \ldots, \\
\gamma_{n-k}^* &\equiv \id \mod \frakm^{n}
\end{split} \]
by applying \cref{lem:commutatormodI} repeatedly. 
In this case, we have $(\gamma_{n-k}^*x - x, \gamma_{n-k}^*y - y) \subset \frakm^n$, 
and thus,  
\[\nu_0(\gamma_{n-k}) 
= \dim \bC[[x,y]]/(\gamma_{n-k}^*x - x, \gamma_{n-k}^*y - y) 
\geq \dim \bC[[x,y]]/\frakm^n
= n(n+1)/2. 
\]
\end{proof}

\subsection{Lattices}
Here, we recall definitions and properties of lattices. 
A \textit{lattice} is a finitely generated free $\bZ$-module $L$ equipped with an inner product, 
that is, a nondegenerate symmetric bilinear form $(\;,\;): L\times L \to \bZ$. 
Let $L = (L, (\;,\;))$ be a lattice of rank $r$, and let 
$e_1, \ldots, e_r$ be a basis of $L$. 
The $r\times r$ matrix $((e_i, e_j))_{ij}$ is called the \textit{Gram matrix} of $L$ 
with respect to the basis $e_1, \ldots, e_r$. 
The determinant $\det ((e_i, e_j))_{ij}\in \bZ\setminus\{0\}$ does not depend on 
the choice of the basis. This integer is called the \textit{determinant} of $L$ and 
is denoted by $\det L$. 
The dual of $L$ is defined to be 
$L^\vee \coloneqq \{ w \in L\otimes\bQ \mid (v,w)\in \bZ \text{ for any $v\in L$} \}$. 
We have $L \subset L^\vee$, and the index $[L^\vee: L]$ is equal to $|\det L|$. 
We say that $L$ is \textit{even} if $(v,v)\in 2\bZ$ for any $v\in L$, and 
\textit{unimodular} if $|\det L| = 1$. 

A \textit{sublattice} of $L$ is a submodule on which the inner product restricts 
to a nondegenerate form. 
Let $M$ be a sublattice of $L$. We define its \textit{orthogonal} by 
$M^\perp \coloneqq \{ w\in L \mid (v,w)=0 \text{ for any $v\in M$} \}$. 
Note that the direct sum $M \oplus M^\perp$ may not coincide with $L$ itself, 
but it is contained in $L$ with finite index. 
Thus, any element $v\in L$ can be uniquely written as 
$v = w + w'$ for some $w\in \bQ M$ and $w'\in \bQ M^\perp$.  

\begin{lemma}\label{lem:detM2-multiple}
Let $M_1$ and $M_2$ be sublattices of a lattice $L$, and assume that 
they are orthogonal to each other and that $M_1\oplus M_2$ is contained in $L$ with 
finite index (e.g. $M_2 = M_1^\perp$). 
Let $v\in L$, and write $v= w_1 + w_2$, where $w_1\in \bQ M_1$ and $w_2 \in \bQ M_2$. 
Then $\det(M_2) w_2 \in M_2$ and $\det(M_2) w_1 \in L$. 
\end{lemma}
\begin{proof}
We have the inclusion 
\[ M_1\oplus M_2 \subset L \subset L^\vee \subset M_1^\vee \oplus M_2^\vee \]
in $L\otimes\bQ$. In particular, $w_2\in M_2^\vee$. This shows that 
$\det(M_2) w_2 = \pm [M_2^\vee:M_2]w_2 \in M_2$. 
Furthermore, we have $\det(M_2) w_1 = \det(M_2) v - \det(M_2) w_2 \in L$. 
\end{proof}

\subsection{Homomorphism $\SL_2(\bR)\to \SO^+(1,2)$}\label{ss:SL_2->S0^+(V)}
For the proof of \cref{th:mainA}, we will make use of a homomorphism $\SL_2(\bR)\to \SO^+(1,2)$
which is well-known in the theory of Lie groups. 

Let $V$ be the $3$-dimensional $\bR$-vector space defined by 
$V = \{ v \in \M_2(\bR) \mid \Tr(v) = 0 \}$, where 
$\M_2(\bR)$ is the set of real $2\times 2$ matrices. Set 
\[ e_1 = \begin{pmatrix}
0 & -1 \\
0 & 0
\end{pmatrix}, \,
e_2 = \begin{pmatrix}
1 & 0 \\
0 & -1
\end{pmatrix}, \,
e_3 = \begin{pmatrix}
0 & 0 \\
1 & 0
\end{pmatrix}. 
\]
Then $e_1, e_2, e_3$ form a basis of $V$. 
We define a symmetric bilinear form $(\;, \;):V\times V \to \bR$ by 
\[ (v,w) = -2\Tr(vw) \quad (v,w\in V). \]
One can check that $(v,v) = 4 \det(v)$ for any $v\in V$. 
The Gram matrix with respect to the basis $e_1, e_2, e_3$ is given by 
\[ \begin{pmatrix}
0 & 0 & 2 \\
0 & -4 & 0 \\
2 & 0 & 0
\end{pmatrix}, 
\]
which implies that this symmetric bilinear form is nondegenerate and 
has signature $(1,2)$. 

Let $\rho:\SL_2(\bR)\to \GL(V)$ be the representation of $\SL_2(\bR)$ on $V$ 
defined by $\rho(A)v = AvA^{-1}$ ($A\in \SL_2(\bR)$, $v\in V$). 
Each $A\in\SL_2(\bR)$ acts isometrically since 
\[ (\rho(A)v, \rho(A)w) = -2\Tr(AvA^{-1}AwA^{-1}) = -2\Tr(vw) = (v,w). \]
Moreover, for $A = \begin{pmatrix}
a & b \\ 
c & d
\end{pmatrix}
\in \SL_2(\bR)$, 
the matrix expression of $\rho(A)$ with respect to the basis $e_1, e_2, e_3$
is given by 
\begin{equation}\label{eq:rho(A)}
\begin{pmatrix}
a^2 & 2ab & b^2 \\
ac & ad + bc & bd \\
c^2 & 2cd & d^2
\end{pmatrix}.
\end{equation} 
In particular, we have $\det(\rho(A)) = 1$, and 
the image of $\rho:\SL_2(\bR)\to \GL(V)$ is contained in $\SO(V)$. By \eqref{eq:rho(A)}, 
it also follows that $\ker(\rho) = \{I_2, -I_2\} \subset \SL_2(\bR)$, 
where $I_n$ denotes the $n\times n$ identity matrix. 
Let $\SO^+(V)$ denote the connected component of $\SO(V)$ containing $I_3$. 
This component $\SO^+(V)$ consists of the isometries in $\SO(V)$ that 
preserve a connected component of the subset $\{ v \in V \mid (v,v) = 1 \} \subset V$. 

\begin{proposition}\label{prop:imageofrho}
The image of $\rho:\SL_2(\bR)\to \GL(V)$ is equal to $\SO^+(V)$. 
\end{proposition}
\begin{proof}
We sketch the proof, as this is well-known. 
One can verify that $\rho:\SL_2(\bR)\to \SO(V)$ is a local diffeomorphism 
by checking that it induces an isomorphism of the corresponding Lie algebras. 
Then, the image of $\rho$ is open in $\SO(V)$, and hence 
coincides with the identity component $\SO^+(V)$. 
\end{proof}

\begin{lemma}\label{lem:chplofrho(A)}
Let $A\in\SL_2(\bR)$. Then, the characteristic polynomial of $\rho(A)$ 
is given by $(t-1)(t^2 - (\Tr(A)^2 - 2)t + 1)$.  
\end{lemma}
\begin{proof}
Let $A = \begin{pmatrix}
a & b \\
c & d
\end{pmatrix}
\in \SL_2(\bR)$. Note that $ad-bc = 1$. 
Let $F(t)\in \bR[t]$ be the characteristic polynomial of $\rho(A)$. 
By \eqref{eq:rho(A)}, the coefficient of $t^2$ is 
\[ -(a^2 + ad + bc + d^2) = -(a^2 + ad + (ad - 1) + d^2) = -((a+d)^2 -1).  \]
On the other hand, $F(t)$ is anti-palindromic since $\rho(A)$ is an isometry of 
an odd-dimensional inner product space with $\det(\rho(A)) = 1$. Thus 
\[ \begin{split}
F(t)
&= t^3 -((a+d)^2 -1) t^2 + ((a+d)^2 -1) t - 1 \\
&= (t-1) (t^2 - ((a+d)^2 -2) t + 1) \\
& = (t-1) (t^2 - (\Tr(A)^2 - 2) t + 1). 
\end{split}
\]
The proof is complete. 
\end{proof}

Let $M = \bZ e_1 + \bZ e_2 + \bZ e_3 \subset V$. Define  
\[ \SO(M) \coloneqq \{ \alpha \in \SO(V) \mid \alpha M\subset M \} 
\quad\text{and} \quad
\SO^+(M) \coloneqq \SO^+(V) \cap \SO(M). \]

\begin{proposition}\label{prop:imageofrho_Z}
We have $\rho(\SL_2(\bZ)) = \SO^+(M)$. 
\end{proposition}
\begin{proof}
Let $A\in\SL_2(\bZ)$. Then, it follows from \eqref{eq:rho(A)} that 
$\rho(A)$ preserves $M$. Furthermore, we have $\rho(A)\in \SO^+(V)$ 
by \cref{prop:imageofrho}. Hence, $\rho(\SL_2(\bZ)) \subset \SO^+(M)$.  

Conversely, let $\alpha \in \SO^+(M)$. By \cref{prop:imageofrho}, there exists 
$A = \begin{pmatrix}
a & b \\
c & d
\end{pmatrix}\in\SL_2(\bR)$ such that $\alpha = \rho(A)$.
Since $\rho(A)$ preserves $M$, all entries of \eqref{eq:rho(A)} are integers. 
In particular, we have $a^2, b^2, c^2, d^2, 2ab, ac, bd, 2cd\in \bZ$, which implies that 
$a,b,c,d\in \sqrt{n}\bZ$ for some square-free $n\in \bZ_{\geq 1}$. 
Since $ad-bc = 1$, it follows that $n=1$, i.e., $A\in\SL_2(\bZ)$. 
This completes the proof. 
We remark that more general arguments can be found in 
\cite[Theorem 7.1]{Dolachev1996} and \cite{Hashimoto-Lee2025}. 
\end{proof}

Let $A\in\SL_2(\bZ)$, and suppose that $|\Tr(A)| > 2$. 
Then, it follows from \cref{lem:chplofrho(A)} that 
$1$ is a simple eigenvalue of $\rho(A)$. 
Since $A (2A - \Tr(A)I_2)A^{-1} = 2A - \Tr(A)I_2$, 
the vector $2A - \Tr(A)I_2\in M$ is an eigenvector of $\rho(A)$ corresponding to
the eigenvalue $1$. 
We say that a vector $v\in M$ is \textit{primitive} (in $M$) if 
the quotient $M/\bZ v$ is torsion-free. 
A primitive eigenvector of $\rho(A)$ corresponding to $1$ is given by 
\[ \frac{1}{\gcd(a-d, 2b, 2c)}\begin{pmatrix}
a-d & 2b \\
2c & -(a-d)
\end{pmatrix}
= \frac{1}{\gcd(a-d, 2b, 2c)}(2A - \Tr(A)I_2)
\]
or its minus, where $A = \begin{pmatrix}
a & b \\ 
c & d
\end{pmatrix}$. 
Hence, if $u$ is a primitive eigenvector of $\rho(A)$ corresponding to $1$, then 
\begin{alignat}{2}
(u,u) 
&= 4\det(u) \notag\\
&= \frac{-4}{\gcd(a-d, 2b, 2c)^2}((a-d)^2 + 4bc) \notag\\
&= \frac{-4}{\gcd(a-d, 2b, 2c)^2}((a-d)^2 + 4ad - 4) \notag\\
&= \frac{-4}{\gcd(a-d, 2b, 2c)^2}(\Tr(A)^2 - 4).  \label{eq:selfIPofu}
\end{alignat}

\section{K3 surfaces and their automorphisms}\label{sec:K3}

\subsection{Definitions and properties}
We recall some definitions and properties of K3 surfaces. See \cite{BHPV} for more details. 
A \textit{K3 surface} is a compact complex surface $X$ such that its canonical line bundle is 
trivial and $H^1(X, \cO_X) = 0$. Let $X$ be a K3 surface. 
The second cohomology group $H^2(X,\bZ)$ equipped with the cup product $(\;, \;)$ is 
an even unimodular lattice of signature $(3,19)$. 
The subspace $H^{1,1}_\bR(X) \coloneqq H^{1,1}(X) \cap H^2(X,\bR)$ of $H^2(X,\bR)$ 
has signature $(1,19)$. The submodule 
$\NS(X) := H^2(X,\bZ)\cap H^{1,1}(X)$ of $H^2(X,\bZ)$ 
is called the \textit{N\'eron-Severi lattice} or the \textit{Picard lattice}. 
For a divisor $D$ of $X$, we refer to the first Chern class of the line bundle 
associated to $D$ as the \textit{class} of $D$ for short.  
A curve in $X$ is called a \textit{$(-2)$-curve} if it is a smooth rational curve 
with self-intersection $-2$. 

\begin{lemma}\label{lem:irrclass}
Let $d\in \NS(X)$ be the class of an irreducible curve $C$. 
Then, we have $(d,d)\geq -2$. Moreover, if $(d,d)=-2$ then $C$ is a $(-2)$-curve. 
\end{lemma}
\begin{proof}
This follows from \cite[Proposition VIII-(3.7)]{BHPV}. 
\end{proof}

We will need the following lemma. 

\begin{lemma}\label{lem:loxodromicand(-2)}
Let $\gamma$ be an automorphism of $X$. 
Suppose that $\gamma^*:H_\bR^{1,1}(X)\to H_\bR^{1,1}(X)$ has an eigenvalue whose 
absolute value is greater than $1$. If $C\subset X$ is an irreducible curve preserved 
by $\gamma$, then $C$ is a $(-2)$-curve. 
\end{lemma}
\begin{proof}
Let $\lambda$ be an eigenvalue of $\gamma^*$ whose absolute value is greater than $1$. 
Then, $\lambda$ is a real number and a simple eigenvalue, and $\lambda^{-1}$ is also 
a simple eigenvalue of $\gamma^*$; see e.g. \cite[Section 3]{McMullen2002}. Set
\[ W = \{ v\in H^{1,1}_\bR(X) \mid 
 ((\gamma^*)^2 - (\lambda + \lambda^{-1})\gamma^* +\id)v = 0 \}. 
\]
Then, $W$ is a nondegenerate $2$-dimensional subspace of signature 
$(1,1)$ which is invariant under $\gamma^*$. 
We remark that $W^\perp \subset H^{1,1}_\bR(X)$ is negative definite 
since the signature of $H^{1,1}_\bR(X)$ is $(1, 19)$. 

Now suppose that $\gamma$ preserves an irreducible curve $C$, and 
let $d\in H^{1,1}_\bR(X)$ be its class. We have $\gamma^*d = d$, which implies that 
$d\in W^\perp$. 
Since $d$ also belongs to $H^{2}(X,\bZ)$, and $H^{2}(X,\bZ)$ is an even lattice, 
we obtain $(d,d)\leq -2$. Hence, $C$ is a $(-2)$-curve by \cref{lem:irrclass}.
\end{proof}

\subsection{K3 surfaces in $\bP^1\times \bP^1 \times \bP^1$}\label{ss:K3inP1P1P1}
Let $X$ be a smooth surface in $Z = \bP^1_x\times \bP^1_y \times \bP^1_z$ defined by 
a tri-homogeneous polynomial of tri-degree $(2,2,2)$. Then $X$ is a K3 surface. 
Let $\pi_x:X\to \bP^1_y\times \bP^1_z$ denote the restriction of the projection 
$Z \to \bP^1_y \times \bP^1_z$. 
Similarly, we define $\pi_y:X\to \bP^1_x\times \bP^1_z$, 
and $\pi_z:X\to \bP^1_x\times \bP^1_y$. 
Then $\pi_x$, $\pi_y$, and $\pi_z$ are double covers and define involutions 
$\iota_x$, $\iota_y$, and $\iota_z:X\to X$, respectively. 

Let $h_x\in \NS(X)$ denote the class of a fiber of the projection $X \to \bP^1_x$. 
We define $h_y$ and $h_z \in \NS(X)$ similarly. 
Let $M \subset \NS(X)$ be the sublattice generated by $h_x$, $h_y$, and $h_z$, i.e.,   
$M \coloneqq \bZ h_x \oplus \bZ h_y \oplus \bZ h_z$. 
In what follows, $M^\perp$ denotes the orthogonal of $M$ in $\NS(X)$, 
that is, $M^\perp = \{ v\in \NS(X) \mid (v, w) = 0 \text{ for any $w\in M$} \}$. 

\begin{lemma}\label{lem:noeffinMperp}
The sublattice $M^\perp$ of $\NS(X)$ contains no effective class. 
In particular, it contains no element with self-inner product $-2$. 
\end{lemma}
\begin{proof}
Suppose that $M^\perp$ contains an effective class $d$ represented by an effective 
divisor $D$. Since $(d, h_x) = 0$ and $(d, h_y) = 0$, the divisor $D$ must be 
contracted by $\pi_z$. Similarly, $D$ is also contracted by $\pi_y$, but it is impossible. 
Hence, $M^\perp$ of $\NS(X)$ contains no effective class.  
In particular, $M^\perp$ contains no element with self-inner product $-2$ because 
if such an element existed, then either it or its minus would be an effective class; 
see \cite[Proposition VIII-(3.7)]{BHPV}.  
\end{proof}

\begin{lemma}\label{lem:Misinvariant}
If $\pi_x$ does not contract any curve then 
$\iota_x^*|_{\bZ h_y \oplus \bZ h_z} = \id_{\bZ h_y \oplus \bZ h_z}$ and 
$\iota_x^*|_{(\bZ h_y \oplus \bZ h_z)^\perp} = -\id_{(\bZ h_y \oplus \bZ h_z)^\perp}$. 
Similar statements hold for $\iota_y$ and $\iota_z$. In particular, the sublattice $M$ is 
preserved by the three involutions $\iota_x^*$, $\iota_y^*$ and $\iota_z^*$, 
provided that none of them contracts any curve.  
\end{lemma}
\begin{proof}
We only show the statement for $\iota_x$, as the arguments for $\iota_y$ and $\iota_z$ 
are entirely analogous. 
It is clear that $\iota_x^* h_y = h_y$ and $\iota_x^* h_z = h_z$. 
Note that $\iota_x^*:H^2(X,\bZ) \to H^2(X,\bZ)$ is an involution, and in particular, 
its eigenvalues are $1$ or $-1$. Therefore, to prove that 
$\iota_x^*|_{(\bZ h_y \oplus \bZ h_z)^\perp} = -\id_{(\bZ h_y \oplus \bZ h_z)^\perp}$, 
it suffices to show that $\Tr(\iota_x^*:H^2(X,\bQ) \to H^2(X,\bQ)) = -18$ since 
$\dim H^2(X,\bQ) = 22$. 

Suppose that $\pi_x$ does not contract any curve. Then, the fixed point set $R\subset X$ of 
$\iota_x$ is isomorphic to the branch curve $B\subset \bP_y \times \bP_z$ via $\pi_x$. 
Since $B$ is a nonsingular curve of bi-degree $(4,4)$, its Euler characteristic can be computed 
to be $-16$. 
On the other hand, by the topological Lefschetz fixed point formula for $\iota_x$, 
the Euler characteristic of $R$ is given by 
\[ \sum_{j=0}^4 (-1)^j \Tr(\iota_x^*:H^i(X,\bQ) \to H^i(X,\bQ)) 
= 2 + \Tr(\iota_x^*:H^2(X,\bQ) \to H^2(X,\bQ)).\]
Hence, we obtain $\Tr(\iota_x^*) = -18$, as required. 
A similar discussion in a more general situation can be found in \cite[Section 4]{Nikulin1983}.
\end{proof}

By \cref{lem:Misinvariant}, if none of $\pi_x$, $\pi_y$, and $\pi_z$ contracts any curve, 
then the matrix expressions of $\iota_x^*|_M$, $\iota_y^*|_M$,
and $\iota_z^*|_M$, with respect to the basis $h_x, h_y, h_z$,  are given by 
\begin{equation}\label{eq:matricesofiotas}
\begin{pmatrix}
-1 & 0 & 0 \\
2 & 1 & 0 \\
2 & 0 & 1 
\end{pmatrix}, 
\begin{pmatrix}
1 & 2 & 0 \\
0 & -1 & 0 \\
0 & 2 & 1 
\end{pmatrix}, \text{ and } 
\begin{pmatrix}
1 & 0 & 2 \\
0 & 1 & 2 \\
0 & 0 & -1 
\end{pmatrix},
\end{equation}
respectively.

\begin{proposition}\label{prop:no_nontrivial_relation}
Suppose that none of $\pi_x$, $\pi_y$, and $\pi_z$ contracts any curve. Then 
\[ \langle \iota_x^*|_M, \iota_y^*|_M, \iota_z^*|_M \rangle 
= \langle \iota_x^*|_M \rangle * \langle \iota_y^*|_M \rangle * \langle \iota_z^*|_M \rangle
\cong (\bZ/2\bZ) * (\bZ/2\bZ) * (\bZ/2\bZ),  
\]
and hence,    
\[ \langle \iota_x, \iota_y, \iota_z \rangle 
= \langle \iota_x \rangle * \langle \iota_y\rangle * \langle \iota_z \rangle
\cong (\bZ/2\bZ) * (\bZ/2\bZ) * (\bZ/2\bZ).
\]   
In particular, $\iota_x\circ\iota_y$ and $\iota_z\circ\iota_x\circ\iota_y\circ\iota_z$ 
satisfy no nontrivial group-theoretic relation. 
\end{proposition}
\begin{proof}
By the matrix expressions \eqref{eq:matricesofiotas}, 
this follows in a similar manner to the argument in \cite[Section 3]{Cantat-Oguiso2015}. 
\end{proof}

With respect to the basis 
$e_1 \coloneqq h_x$, $e_2 \coloneqq -h_x + h_y - h_z$, $e_3 \coloneqq h_z$, 
the Gram matrix of $M$ is given by 
\[ \begin{pmatrix}
0 & 0 & 2 \\
0 & -4 & 0 \\
2 & 0 & 0
\end{pmatrix}.  
\]
Thus, we identify this lattice with the lattice 
$M\subset V = \{ v \in \M_2(\bR) \mid \Tr(v) = 0 \}$ 
introduced in \S\ref{ss:SL_2->S0^+(V)}. 
Let $\rho:\SL_2(\bZ) \to \SO^+(M)$ be the homomorphism described there. 
We make use of this homomorphism for the computation below. 

Set $\tilde\phi := \iota_x\circ \iota_y$ and 
$\psi := \iota_z\circ \iota_x\circ \iota_y\circ \iota_z$, and 
assume that none of $\pi_x$, $\pi_y$, and $\pi_z$ contracts any curve.
By \eqref{eq:matricesofiotas}, one can see that 
all eigenvalues of $\tilde\phi^*|_M$ are $1$. 
The same holds for $\psi^*|_M$ since it is the conjugate of $\tilde\phi^*|_M$ by $\iota_z^*|_M$. 
In particular, since $\det(\tilde\phi^*|_M) = 1$ and $\det(\psi^*|_M) = 1$ 
and every automorphism preserves the ample cone, it follows that 
$\tilde\phi^*|_M$ and $\psi^*|_M$ belong to $\SO^+(M)$.  
 
Thus, by \cref{prop:imageofrho_Z}, there exists $\tilde\Phi\in \SL_2(\bZ)$ such that 
$\rho(\tilde\Phi) = \tilde\phi^*|_M$. 
Moreover, by \cref{lem:chplofrho(A)}, we have $\Tr(\tilde\Phi) = 2$ or $-2$. 
We may assume that $\Tr(\tilde\Phi) = 2$ by replacing $\tilde\Phi$ with $-\tilde\Phi$ 
if necessary. Then, there exists $P\in \SL_2(\bZ)$ and $s\in \bZ\setminus\{0\}$ such that 
\[
P^{-1}\tilde\Phi P = \begin{pmatrix}
1 & s \\
0 & 1
\end{pmatrix}.
\]
We may assume $\tilde\Phi = \begin{pmatrix}
1 & s \\
0 & 1
\end{pmatrix}$
by replacing $\rho$ with $\rho(P\bullet P^{-1})$. 
Recall that $M^\perp$ is the orthogonal of $M$ in $\NS(X)$. 

\begin{proposition}\label{prop:nofixedcurve}
Suppose that none of $\pi_x$, $\pi_y$, and $\pi_z$ contracts any curve. 
For any $N> 2|\det(M^\perp)|$ and $i\geq 1$, 
$\gamma_i$ preserves no irreducible curve, where  
$\gamma_i$'s ($i\geq 0$) are the automorphisms defined recursively by 
\[
\gamma_0 \coloneqq \psi, \quad 
\gamma_{i} \coloneqq [\tilde\phi^N, \gamma_{i-1}] \quad(i\geq 1). 
\] 
\end{proposition}
\begin{proof}
Put $D = |\det(M^\perp)|$, and let $N> 2D$ and $i\geq 1$.  
Let $\Gamma_{i-1}\in \SL_2(\bZ)$ be an element such that 
$\rho(\Gamma_{i-1}) = \gamma_{i-1}^*|_M$, and write 
$\Gamma_{i-1} = \begin{pmatrix}
a & b \\ 
c & d
\end{pmatrix}$, 
where $a$, $b$, $c$, and $d$ are integers with $ad - bc = 1$. 
Note that $c\neq 0$; otherwise $\Gamma_{i-1}$ and $\tilde\Phi$ would commute, 
and hence $\gamma_{i-1}^*|_M$ and $\tilde\phi^*|_M$ would also commute, which contradicts 
\cref{prop:no_nontrivial_relation}.  
Put $\Gamma_i = [\Gamma_{i-1}^{-1}, \tilde\Phi^{-N}]$. Then 
\[ \begin{split}
\rho(\Gamma_{i}) 
&= \rho(\Gamma_{i-1} \tilde\Phi^N \Gamma_{i-1}^{-1} \tilde\Phi^{-N}) \\
&= \gamma_{i-1}^*|_M \circ (\tilde\phi^*|_M)^N \circ 
  (\gamma_{i-1}^*|_M)^{-1} \circ (\tilde\phi^*|_M)^{-N} \\
&= (\tilde\phi^{-N} \circ \gamma_{i-1}^{-1} \circ \tilde\phi^N \circ \gamma_{i-1} )^*|_M\\
& = \gamma_i^*|_M, 
\end{split}
\]
and 
\[ \begin{split}
\Gamma_{i} 
&= \Gamma_{i-1} \tilde\Phi^N \Gamma_{i-1}^{-1} \tilde\Phi^{-N} \\  
&= \begin{pmatrix}
a & b \\
c & d
\end{pmatrix}
\begin{pmatrix}
1 & Ns \\
0 & 1
\end{pmatrix}
\begin{pmatrix}
d & -b \\
-c & a
\end{pmatrix}
\begin{pmatrix}
1 & -Ns \\
0 & 1
\end{pmatrix}
\\
&= \begin{pmatrix}
1-Nsac & - Ns + Nsa^2 + N^2s^2ac \\
-Nsc^2 & 1 + Nsac + N^2s^2c^2
\end{pmatrix}. 
\end{split}
\]
In particular, we have $\Tr(\Gamma_i) = 2 + N^2s^2c^2 > 2$, and $1$ is a simple eigenvalue 
of $\gamma_i^*|_M = \rho(\Gamma_i)$. Let $u\in M$ be a primitive eigenvector of 
$\rho(\Gamma_i)$ corresponding to the eigenvalue $1$. By \eqref{eq:selfIPofu}, we have 
\begin{alignat}{2}
- (u,u) 
&= 4 \gcd(-2Nsac - N^2s^2c^2, -2(Ns - Nsa^2 - N^2s^2ac), -2Nsc^2)^{-2}(\Tr(\Gamma_i)^2 - 4)& \notag\\
&\geq 4 (2Nsc^2)^{-2}((2 + N^2s^2c^2)^2 - 4)& \notag\\
&= (Nsc^2)^{-2}(4 N^2s^2c^2 + N^4 s^4 c^4)& \notag\\
&\geq  N^2 s^2& \notag\\  
&> 4 D^2.& \label{eq:-(u,u)>4D^2}   
\end{alignat}
The inequality $\Tr(\Gamma_i) > 2$ also implies that 
$\gamma_i^*|_M$ has a real eigenvalue greater than $1$ by \cref{lem:chplofrho(A)}. 
Hence, by \cref{lem:loxodromicand(-2)}, if there exists a 
$\gamma_i$-invariant irreducible curve, it must be a $(-2)$-curve. 

Now suppose that $\gamma_i$ preserves a $(-2)$-curve $E\subset X$. 
Let $e \in \NS(X)$ be the class of $E$. 
Since $e$ is fixed by $\gamma_i^*$, it can be written as 
$e = \alpha u + v$, where $\alpha\in \bQ$ and $v\in \bQ M^\perp$. 
By \cref{lem:noeffinMperp}, we have $\alpha\neq 0$. Furthermore, 
by \cref{lem:detM2-multiple}, we obtain $D\alpha\in \bZ$ since $u$
is primitive in $L$. Then, we obtain  
\[ -2D^2 
= (D e,D e) 
= (D\alpha u, D\alpha u) + (Dv, Dv) 
\leq (D\alpha)^2(u,u) \leq (u,u),  \]
which contradicts \eqref{eq:-(u,u)>4D^2}. 
Therefore, $\gamma_i$ does not preserve any irreducible curve.  
\end{proof}

\section{Main results}\label{sec:mainresults}
This section gives the proofs of \cref{th:mainA,th:mainB}. 
Let $X\subset Z \coloneqq \bP^1_x\times \bP^1_y \times \bP^1_z$ 
be the surface defined by the tri-homogeneous polynomial 
\[\begin{split}
F(x_0, x_1, y_0, y_1, z_0, z_1) = Q_a(x_0, x_1, y_0, y_1) z_0^2
+ (x_0^2 y_0^2 + x_1^2 y_1^2)z_0 z_1 
+ Q_b(x_0, x_1, y_0, y_1) z_1^2, 
\end{split} 
\]
where $a,b \in \bC$, and
$Q_c(x_0, x_1, y_0, y_1) \coloneqq x_0^2y_1^2 + c x_0x_1y_0y_1 + x_1^2y_0^2$
for $c\in \bC$. The case $(a,b) = (0,1)$ corresponds to $X$ in \cref{th:mainA}. 

\begin{proposition}\label{prop:Xissmooth}
The surface $X$ is smooth if and only if
\begin{equation}\label{eq:Xissmooth}
(a^2-4)(b^2-4)(ab-2b-2a+3)(ab+2b+2a+3)(b^2-2ab+a^2+4) \neq 0.
\end{equation} 
\end{proposition}
\begin{proof}
This proposition follows from the Jacobian criterion, in principle. 
Due to the complexity of the manual computations, we used the computer algebra system 
Singular \cite{Singular} to verify it.
\end{proof}

In what follows, we assume that $a$ and $b$ satisfy \eqref{eq:Xissmooth}, that is,  
$X$ is smooth. Let $p$ and $q\in Z$ be the points given by 
$(x_0:x_1, y_0:y_1, z_0:z_1) = (1:0, 1:0, 1:0)$ and 
$(x_0:x_1, y_0:y_1, z_0:z_1) = (1:0, 1:0, 0:1)$, respectively. 
These points lie on $X$. Because 
\[
F(x_0,x_1,1,0,1,0) = x_1^2, \quad
F(x_0,x_1,1,0,0,1) = x_1^2,
\]
the points $p$ and $q$ are fixed by $\iota_x$. 
Similarly, these two points are also fixed points of $\iota_y$. 
Furthermore, it follows from  
\[ F(1,0,1,0,z_0,z_1) = z_0 z_1 \]
that $\iota_z(p) = q$ and $\iota_z(q) = p$. 

The idea of the proof of \cref{th:mainA} is as follows. 
Let $\phi$ be a suitable power of $\iota_x\circ\iota_y$, and let  
$\psi = \iota_z\circ\iota_x\circ\iota_y\circ \iota_z$. 
These two automorphisms fix the point $p$. 
We will choose the parameters $a$ and $b$ so that $\phi$ and $\psi$ 
satisfy no nontrivial group-theoretic relation and the induced homomorphism of $\phi$ 
on $\widehat\cO_{p}$ is congruent to $\id$ modulo $\frakm_{p}^2$. 
Then, we apply \cref{prop:idmodm_n} to $\phi$ and $\psi$ to deduce \cref{th:mainA}. 

\begin{lemma}\label{lem:nocontraction}
If $a^2 -ab + b^2 \neq 0$, then neither $\pi_x$ nor $\pi_y$ contracts any curve. 
If $a\neq b$, then $\pi_z$ does not contract any curve. 
\end{lemma}
\begin{proof}
We only show that $\pi_x$ does not contract any curve if $a^2 -ab + b^2 \neq 0$, because 
the proof for $\pi_y$ is the same, and the one for $\pi_z$ is simpler. 
We prove the contrapositive. Suppose that $\pi_x$ contracts a curve. 
This means that there exist $(y_0:y_1)\in \bP_y$ and $(z_0:z_1)\in \bP_z$ such that 
all three coefficients of $F$, viewed as a quadratic form in $x_0$ and $x_1$, 
vanish simultaneously:  
\begin{alignat}{3}
&y_1^2 z_0^2 + y_0^2 z_0 z_1 + y_1^2 z_1^2 &&= 0  \label{eq:s1}\\
&y_0 y_1 (a z_0^2 + b z_1^2) &&= 0  \label{eq:s2}\\ 
&y_0^2 z_0^2 + y_1^2 z_0 z_1 + y_0^2 z_1^2 &&= 0.  \label{eq:s3}
\end{alignat}
We have $y_0 \neq 0$; otherwise, \eqref{eq:s1} and \eqref{eq:s3} would imply that 
$z_0 = 0$ and $z_1 = 0$. Similarly, we have $z_0 \neq 0$. 
Hence, we may assume that $y_0 = 1$ and $z_0 = 1$. 
Then, \eqref{eq:s1} and \eqref{eq:s3} imply that $y_1 \neq 0$ and 
$y_1^2 = z_1^3$. Hence 
\begin{equation*}
0 = 1 + y_1^2 z_1 + z_1^2 = 1 + z_1^2 + z_1^4.
\end{equation*} 
On the other hand, it follows from \eqref{eq:s2} that 
$a + b z_1^2 = 0$. Therefore  
\[ b^2 -ab + a^2 
= b^2 - (-b z_1^2)b  + (-b z_1^2)^2
= b^2(1 + z_1^2 + z_1^4) 
= 0. 
\]
This completes the proof. 
\end{proof}

We examine the behavior of $\iota_x \circ \iota_y$ at the point $p$. 
Put $U \coloneqq (x_0\neq 0)\cap(y_0\neq 0)\cap(z_0\neq 0) \subset Z$, and 
define $x \coloneqq x_1/x_0$, $y \coloneqq y_1/y_0$, and $z \coloneqq z_1/z_0$. 
Then $(U; x, y, z)$ is a chart around $p$, which corresponds to $(x, y, z) = (0,0,0)$. 
The defining equation $F$ of $X$ is written as 
\[ \begin{split}
f(x, y, z) &=  (y^2 + b xy + x^2) z^2 + (1 + x^2 y^2) z + (y^2 + a x y + x^2) \\
& = (z^2 + y^2 z + 1) x^2 + (b y z^2 + a y) x + (y^2 z^2 + z + y^2) \\
& = (z^2 + x^2 z + 1) y^2 + (b x z^2 + a x) y + (x^2 z^2 + z + x^2)
\end{split} \]
on $U$. By Vieta's formulas, we have
\begin{equation}\label{eq:iotaonU}
\iota_x(x,y,z) = \left(- x - \frac{(b z^2 + a)y}{z^2 + y^2 z + 1}, y, z \right), \quad
\iota_y(x, y, z) = \left(x, -y - \frac{(b z^2 + a)x}{z^2 + x^2 z + 1}, z \right) 
     \end{equation}
on (the intersection of $X$ and) an open subset of $U$. 

Since $\frac{\partial f}{\partial z}(0,0,0) = 1 \neq 0$, the implicit function theorem 
implies that $(x,y)$ defines local coordinates of $X$ around $p$. 
We identify $\widehat\cO_p$ with $\bC[[x,y]]$. Then  
$x,y$ form a basis of $\frakm_p/\frakm_p^2$. Note that 
$\iota_x$ and $\iota_y$ induce linear transformations on $\frakm_p/\frakm_p^2$ since 
they are automorphisms fixing the point $p$.  

\begin{lemma}\label{lem:actiononm_p/m_p^2}
The actions $(\iota_x)^*_p$ and $(\iota_y)^*_p$ on $\frakm_p/\frakm_p^2$
are represented, with respect to the basis $x, y$,  by the matrices 
\[ \begin{pmatrix}
-1 & 0 \\
-a & 1
\end{pmatrix} \quad\text{and}\quad 
\begin{pmatrix}
1 & -a \\
0 & -1
\end{pmatrix}, 
\]
respectively. 
\end{lemma}
\begin{proof}
As a function in $(x, y)$, we have $z(0,0) = 0$. Hence, by \eqref{eq:iotaonU}, 
we obtain $(\iota_x)_p^*x = -x - a y$, $(\iota_x)_p^*y = y$, 
$(\iota_y)_p^*x = x$, and $(\iota_y)_p^*y = -ax -y$ modulo $\frakm_p^2$. 
This proves the claim.
\end{proof}

\begin{proposition}\label{prop:finiteorderlinearpart_phi}
If $a^2 = 0$, $1$, $2$, or $3$ then the order of 
$(\iota_x\circ \iota_y)^*_p: \frakm_p/\frakm_p^2 \to \frakm_p/\frakm_p^2$ is $2$, $3$, $4$, 
or $6$, respectively. These are the only values of $a$ for which 
$(\iota_x\circ \iota_y)^*_p:\frakm_p/\frakm_p^2 \to \frakm_p/\frakm_p^2$ is of finite order. 
\end{proposition}
\begin{proof}
The linear transformation
$(\iota_x\circ \iota_y)^*_p = (\iota_y)^*_p\circ (\iota_x)^*_p$ can be represented by 
\[ \begin{pmatrix}
-1 + a^2 & -a \\
a & -1
\end{pmatrix} \]
by \cref{lem:actiononm_p/m_p^2}. The first statement follows by direct computations.  
If $(\iota_x\circ \iota_y)^*_p$ has finite order, then its characteristic polynomial
$t^2 - (a^2 - 2) t + 1$ is a product of cyclotomic polynomials. This implies that 
$a^2 = 0$, $1$, $2$, $3$ or $4$. However, $a^2 \neq 4$ by the assumption \eqref{eq:Xissmooth}. 
This completes the proof.
\end{proof}

We are now ready to prove the main theorems. 

\begin{proof}[Proof of \cref{th:mainA}]
We are in the case $(a,b) = (0,1)$. By \cref{lem:nocontraction}, 
the double covers $\pi_x$, $\pi_y$, and $\pi_z$ do not contract any curve. 
Thus, by \cref{prop:no_nontrivial_relation}, $\iota_x\circ \iota_y$ and 
$\iota_z \circ \iota_x \circ \iota_y \circ \iota_z$ satisfy no nontrivial 
group-theoretic relation. 
Let $M\subset \NS(X)$ be the sublattice of rank $3$ defined as in \S\ref{ss:K3inP1P1P1}, 
and put $D = |\det(M^\perp)|$. Set 
$\phi = (\iota_x\circ \iota_y)^{2D+2}$ and
$\psi = \iota_z \circ \iota_x\circ \iota_y \circ \iota_z$. 
We recursively define automorphisms $\gamma_i$ of $X$ by 
$\gamma_0 \coloneqq \psi$ and $\gamma_{i} \coloneqq [\phi, \gamma_{i-1}]$ ($i\geq 1$).  

Now, let $n\geq 1$ be any integer. 
By \cref{prop:nofixedcurve}, no irreducible curve is preserved 
by $\gamma_n$, and in particular, the point $p$ is an isolated fixed point of $\gamma_n$.  
Moreover, \cref{prop:finiteorderlinearpart_phi} gives 
$\phi^*_p \equiv \id \bmod \frakm_p^2$. 
We also have $\psi^*_p \equiv \id \bmod \frakm_p$; indeed, this congruence 
holds for any automorphism fixing $p$. 
Applying \cref{prop:idmodm_n} with $k=1$, we obtain 
$(\gamma_{n})^*_p = (\gamma_{(n+1)-1})^*_p \equiv \id \bmod \frakm_p^{n+1}$, and in particular,  
$(\gamma_{n})^*_p \equiv \id \bmod \frakm_p^{n}$.   
Therefore, $\gamma_n$ is the desired automorphism, and the proof is complete. 
\end{proof}

\begin{remark}\label{rmk:parameter}
The parameters $a$ and $b$ can be chosen differently. In fact, a similar argument 
applies as long as $a$ and $b$ satisfy the following conditions: 
\begin{enumerate}
\item Condition \eqref{eq:Xissmooth} holds so that $X$ is smooth; 
\item $a^2 - ab + b^2 \neq 0$ and $a\neq b$ so that none of $\pi_x$, $\pi_y$, and 
$\pi_z$ contracts any curve; 
\item $a^2 \in \{ 0,1,2,3 \}$ so that a suitable power of 
$(\iota_x\circ \iota_y)^*_p$ is congruent to $\id$ modulo $\frakm_p^2$.
\end{enumerate}
For example, when $a=0$, the same proof goes through as long as   
\[ b\in \bC\setminus\{ -2,-3/2,0,3/2,2, -2\sqrt{-1}, 2\sqrt{-1} \}.\]
\end{remark}

\begin{proof}[Proof of \cref{th:mainB}]
Let $n\geq1$ be an integer, and let $\gamma$ be the automorphism described in \cref{th:mainA}. 
Then, $C' \coloneqq \gamma^{-1}(C)$ is isomorphic but not equal to $C$ since $\gamma$ does not 
preserve any irreducible curve. 

Let $(x,y)$ be a local coordinate system around $p$. 
We may assume that $C$ is defined by the equation $y=0$ on this coordinate system
since $C$ is nonsingular. Write $\gamma(x,y) = (f(x,y), g(x,y))$ on this  
coordinate system, where 
$f, g\in\cO_p$. Then $g$ can be written as $g(x, y) = y + g'(x,y)$
for some $g'\in \frakm_p^n$ since $\gamma^*_p \equiv \id \bmod \frakm_p^n$.  
Hence we have 
\[ \begin{split}
\mu_p(C, C') 
&= \dim \bC[[x,y]]/(y, \gamma^*_p y) \\
&= \dim \bC[[x,y]]/(y, g(x,y)) \\
&= \dim \bC[[x]]/(g(x,0)) \\
&= \dim \bC[[x]]/(g'(x,0)) \\
&\geq n. 
\end{split}
\]
The proof is complete.
\end{proof}

\bibliographystyle{plain}
\bibliography{refs} 
\end{document}